\documentclass[psamsfonts,12pt]{amsart}
\usepackage[left=1in,right=1in,top=1.00in,bottom=1.00in]{geometry}
\usepackage{amssymb,mathrsfs}
\newtheorem{theorem}{Theorem}[section]

\newtheorem{corollary}[theorem]{Corollary}
\newtheorem{proposition}[theorem]{Proposition}
\newtheorem{fact}[theorem]{Fact}

\theoremstyle{definition}

\theoremstyle{remark}
\newtheorem{claim}{Claim}

\newcommand\N{\mathbb{N}}
\newcommand\Z{\mathbb{Z}}
\newcommand\T{\mathbb{T}}
\newcommand\cont{\mathfrak{c}}
\newcommand\BS{\Z^\N}
\newcommand\B{X}
\newcommand\Hom{\mathrm{Hom}}
\newcommand\G{\Hom(\B,\Z)}

\newcommand\du[1]{\widehat{#1}}

\newcommand\dual[1]{\widehat{#1}}

\title{A countable free closed non-reflexive subgroup of $\Z^\cont$}

\author[M. Ferrer]{Maria Vincenta Ferrer}
\address[Maria Vincenta Ferrer]{Universitat Jaume I, Instituto de Matem\'aticas de Castell\'on,
Campus de Riu Sec, 12071 Castell\'{o}n, Spain}
\email{mferrer@mat.uji.es}

\author[S. Hern\'andez]{Salvador Hern\'andez}
\address[Salvador Hern\'andez]{Universitat Jaume I, INIT and Departamento de Matem\'{a}ticas,
Campus de Riu Sec, 12071 Castell\'{o}n, Spain}
\email{hernande@mat.uji.es}

\author[D. Shakhmatov]{Dmitri Shakhmatov}
\address[Dmitri Shakhmatov]{Division of Mathematics, Physics and Earth Sciences\\
Graduate School of Science and Engineering\\
Ehime University, Matsuyama 790-8577, Japan}
\email{dmitri.shakhmatov@ehime-u.ac.jp}



\thanks{The first and second listed authors acknowledge partial support by the Generalitat Valenciana,
grant code: PROMETEO II/2014/062; and by Universitat Jaume I,
grant P1$\cdot$1B2015-77}

\thanks{The third listed author was partially supported by the Grant-in-Aid for Scientific Research~(C) No.~26400091 by the Japan Society for the Promotion of Science (JSPS)}

\begin{document}
\begin{abstract}
We prove that the group $G=\Hom(\BS, \Z)$ of all homomorphisms from the Baer-Specker group $\BS$ to the group $\Z$ of integer numbers endowed with the topology of pointwise convergence contains no infinite compact subsets.
We deduce from this fact that the second Pontryagin dual of $G$ 
is discrete. As $G$ is non-discrete, it is not reflexive. Since 
$G$ can be viewed as a closed subgroup of the Tychonoff product $\Z^\cont$ of continuum many copies of the integers $\Z$, this provides 
an example of a group described in the title, thereby answering Problem 11 from
[J.~Galindo, L.~Recorder-N\'u\~{n}ez, M.~Tkachenko, 
{\em Reflexivity of prodiscrete topological groups\/}, J. Math. Anal. Appl. 384 (2011), 320--330]. 
It follows that an inverse limit of finitely generated (torsion-)free discrete abelian groups
need not be reflexive.
\end{abstract}
\maketitle

\section{Introduction}

The circle group $\T$ is the quotient of the real numbers over the subgroup 
of integer numbers. (Alternatively, one can view $\T$ as the unit circle in the complex plane with multiplication as a group operation.)
A {\em character\/} of a topological group $G$ is a continuous group homomorphism
from $G$ to the circle group $\T$. The set of all characters of an abelian topological group $G$ forms an abelian group $\du{G}$ with pointwise addition as its group operation. When equipped with the compact-open topology, $\du{G}$ becomes a topological group called the (Pontryagin) {\em dual group\/} of $G$. 

We say that an abelian  topological group $G$ is {\em reflexive\/} provided that the second Pontryagin dual $\du{\du{G}}$ of $G$ is topologically isomorphic to $G$. The celebrated theorem of Pontryagin-van Kampen says that all locally compact abelian groups are reflexive.  
Understanding the duality features of large classes of topological abelian groups beyond locally compact ones is a major topic of the topological group theory.

Extending the scope of the Pontryagin-van Kampen theorem, Banasczyk \cite{Ban} introduced 
the class of nuclear groups and investigated its duality properties.
The class of nuclear groups contains all locally compact abelian groups 
and is closed under taking products, subgroups and quotients \cite[7.5 and 7.10]{Ban}. It follows from this that every prodiscrete abelian group is nuclear.
(Recall that a topological group is {\em prodiscrete\/} if it is a closed subgroup of some product of discrete groups.)

Hofmann and Morris \cite[Problem 5.3]{HM-TP} propose to develop both a structure and a character theory of prodiscrete groups. (Recently, they re-phrased 
this problem 
in \cite[Question 1]{HM-Axioms} asking if there exists a satisfactory structure theory even for abelian non-discrete prodiscrete groups.)
In item (a) of \cite[Problem 5.3]{HM-TP}
the special case of compact-free prodiscrete groups is considered, and as a typical example, the group $G=\Hom(\BS,\Z)$ is proposed, and it is explicitly mentioned that the dual and bi-dual of this group has not been investigated.
We show that the 
bi-dual $\du{\du{G}}$ of $G$ is the group itself endowed with the discrete topology. Since $G$ is non-discrete, it follows that $G$ is non-reflexive.

Galindo, Recorder-N\'u\~{n}ez and Tkachenko \cite{GRT} found many examples of non-reflexive prodiscrete groups. Indeed, they show that any bounded 
torsion abelian group of cardinality $\cont$ admits a Hausdorff group topology making it into a prodiscrete non-reflexive group of countable pseudocharacter.
(Recall that a topological group $G$ has {\em countable pseudocharacter\/} if its identity is an intersection of a countable family of open subsets of $G$.) 
In \cite[Problem 11]{GRT}, Galindo, Recorder-N\'u\~{n}ez and Tkachenko
ask if 
$\Z^\cont$ contains a closed subgroup of countable pseudocharacter which is not  reflexive.

Note that
$G=\Hom(\BS,\Z)$
can be viewed as a closed subgroup of $\Z^\cont$. Furthermore, it is well known that $G$ is algebraically isomorphic to the free abelian group with countably many generators (see Fact \ref{G:is:free}); in particular, $G$ is countable. 
Since countable groups have countable pseudocharacter, our example $G$ gives a strong positive answer to the problem of Galindo, Recorder-N\'u\~{n}ez and Tkachenko.

Banasczyk \cite[Remark 17.14]{Ban}  says that it remains an open question whether closed subgroups or Hausdorff quotients  of uncountable products of $\mathbb{R}$'s and $\Z$'s are reflexive. 
Answering half of this question, Au\ss enhofer \cite{Aussenhofer} found a non-reflexive Hausdorff quotient of
$\Z^\cont$. We answer another half of this question by 
noticing that $G=\Hom(\BS,\Z)$
is a (countable)
closed non-reflexive subgroup of $\Z^\cont$.

Negrepontis \cite[Corollary 4.8]{Neg} claims that direct and inverse limits of compactly generated locally compact abelian groups are reflexive.
Banasczyk \cite[Remark 17.14]{Ban} points out an error in his proof.
Finally, our counter-example shows that not only the proof but also the statement of the inverse limit theorem from \cite{Neg} is wrong. Indeed, 
our closed
non-reflexive 
subgroup 
$G$
of 
$\Z^\cont$ is an inverse limit of discrete (thus, locally compact) finitely generated (thus, compactly generated) free abelian groups; see Corollary \ref{3.7}.

\section{The Baer-Specker group and the group of its integer-valued homomorphisms}

We denote by $\Z$ the group of integer numbers equipped with the discrete topology.
The group 
$$
\B=\BS
$$ 
is called 
the {\em Baer-Specker group\/} \cite{Baer,Specker}.
We consider 
the Baer-Specker group $\B=\BS$ with the Tychonoff product topology.

For abelian groups $A$ and $B$ we denote by $\Hom(A,B)$ the group of all homomorphisms from $A$ to $B$. 
Let 
$$
G=\G
$$ 
be the group of all group homomorphisms from the Baer-Specker group $\BS$ to the group $\Z$ of integer numbers. 
Clearly, $G$ can be viewed as a subgroup of the direct product $\Z^{\B}$.
One easily checks that $G$ is a closed subgroup of $\Z^{\B}$ when the latter group is equipped with the Tychonoff product topology.
This topology induces on $G$ the {\em topology of pointwise convergence\/}
which we consider in this paper.

\begin{fact}
\label{claim:Specker}
\label{aumatic:continuity}
Let $g\in G=\G$. Then:
\begin{itemize}
\item[(i)] there exist an integer $n\in\N$ and a homomorphism $g'\in\Hom(\Z^n,\Z)$ such that $g=g'\circ \pi_n$, where $\pi_n:\Z^\N\to \Z^n$ is the projection on the first $n$ coordinates;
\item[(ii)] $g$ is continuous.
\end{itemize}
\end{fact}
\begin{proof}
The proof of item (i) can be found in \cite{Specker}.

Item (ii) easily follows from item (i). Indeed,
let $n$ and $g'$ be as in item (i).
Since 
$\Z^n$ is discrete, 
$g'$ is continuous. Since $\pi_n$ is also continuous,
the continuity of $g$ follows. 

Another proof of item (ii) 
can be found in \cite{Nunke}. 
\end{proof}

The following fact is well known; see  for example, 1.3 and 2.4 in Chapter III of \cite{EM}. It also follows easily from Fact \ref{claim:Specker}~(i).

\begin{fact}
\label{G:is:free}
For every $n\in\N$, let $g_n: \BS\to \Z$ denote the projection on $n$th coordinate defined by $g_n(x)=x(n)$ for $x\in \BS$.
Then $\{g_n:n\in \N\}$ is an independent set of generators of $G$; 
in particular, $G$ is the free abelian group with countably many generators.
\end{fact}

The next proposition is probably known. We present its proof for the reader's convenience only.

\begin{proposition}
\label{G:is:not:discrete}
$G$ is non-discrete.
\end{proposition}
\begin{proof}
Let $Y$ be a finite subset of $\B$. Consider a basic open neighbourhood 
\begin{equation}
\label{U_Y}
U_Y=\{g\in G: g(y)=0
\mbox{ for all }
y\in Y\}
\end{equation}
of $0$ in $G$. It suffices to show that $U_Y\not=\{0\}$.
By the well-known property of the Baer-Specker group, there exists a 
finitely generated subgroup $H$ of $X$ containing $Y$ 
and a subgroup $N$ of $X$ such that $X=H\oplus N$; see, for example, \cite[proof of (19.2)]{Fuchs}.
Since $X$ is not finitely generated, $N\not=\{0\}$.
Therefore, there exists a non-trivial homomorphism $f: N\to\Z$.
(One can take as $f$ the restriction to $N$ of a suitable projection from $\B=\Z^\N$ to $\Z$.)
Now let $g\in G$ be such that that $g\restriction_H=0$ and $g\restriction_N=f$.
Since $Y\subseteq H$, one has $g\in U_Y$. On the other hand, $g\not=0$ as $g\restriction_N=f\not=0$.
\end{proof}

\section{Results}

Our principal result is the following

\begin{theorem}
\label{main:theorem}
All compact subsets of the group $G=\G$ are finite.
\end{theorem}

Theorem 
\ref{main:theorem} 
enables us to compute the second dual group $\du{\du{G}}$ of $G$.

\begin{corollary}
\label{second:dual}
$\du{\du{G}}$  is discrete and algebraically isomorphic to $G$.
\end{corollary}
\begin{proof}
Since compact subsets of $G$ are finite by Theorem \ref{main:theorem}, the dual group 
$\dual{G}$ has the topology of pointwise convergence on $G$; that is, $\dual{G}$ can be 
identified with the subgroup of $\T^G$ endowed with the Tychonoff product topology. Since $G$ is countable by Fact \ref{G:is:free}, the dual group $\dual{G}$ of $G$ is (precompact and) metrizable. Therefore, the second dual $\dual{\dual{G}}$ of $G$ 
is a $k$-space; see \cite{Chasco, Au}.

Since $G$ is prodiscrete, the natural map $\alpha_G:G\to \du{\du{G}}$ from $G$ to its second dual $\du{\du{G}}$ is both open and an algebraic isomorphism between $G$ and  $\du{\du{G}}$; see \cite[Theorem 3.1(2)]{GRT}. Therefore, its inverse $\alpha_G^{-1}:\du{\du{G}}\to G$ is a continuous isomorphism from $\du{\du{G}}$ onto $G$. Since all compact subsets 
of $G$ are finite by Theorem \ref{main:theorem}, the same holds for $\du{\du{G}}$. Since $\du{\du{G}}$ is a $k$-space, it must be discrete.
\end{proof} 

Proposition \ref{G:is:not:discrete} and Corollary \ref{second:dual} imply 
the following
\begin{corollary}
$G$ is non-reflexive.
\end{corollary}

From Fact \ref{G:is:free} and this corollary we get the following

\begin{corollary}
\label{cor:3.4}
$G$ is a non-reflexive closed free subgroup of $\Z^\B$ with countably many generators.
\end{corollary}

Since $|\B|=|\BS|=\cont$, the topological groups $\Z^{\B}$ and $\Z^\cont$ are topologically isomorphic.
Thus, we obtain the following 
\begin{corollary}
\label{Z^c}
$\Z^\cont$ contains a countable closed non-reflexive subgroup. 
\end{corollary}

In \cite[Problem 11]{GRT}, Galindo, Recorder-N\'u\~{n}ez and Tkachenko
ask if 
$\Z^\cont$ contains a closed subgroup of countable pseudocharacter which is not  reflexive.
Since countable topological groups have countable pseudocharacter,
this corollary provides a strong positive answer to 
this problem. 
Earlier, 
Pol and Smentek \cite[6.3]{PS} constructed a countable closed non-reflexive subgroup 
in the power $A^\cont$ of the free abelian group $A$ with countably many generators equipped with the discrete topology.

Recall that a topological abelian group $G$ is called {\em strongly reflexive\/} if all closed subgroups and all quotient groups of $G$ are reflexive; see, for example, \cite{CDMP}.

Corollary \ref{Z^c} implies the main result of Au\ss enhofer from \cite{Aussenhofer}:
\begin{corollary}
{\rm (\cite[Corollary 3.6]{Aussenhofer})}
$\Z^\cont$ is not strongly reflexive.
\end{corollary}

Our last corollary provides a counter-example to \cite[Corollary 4.8]{Neg}.

\begin{corollary}
\label{3.7}
An inverse limit of finitely generated (torsion-)free discrete abelian groups
need not be reflexive.
\end{corollary}
\begin{proof}
Let $G$ be as in Corollary \ref{cor:3.4}.
For every finite set $A\subseteq X$ let $\pi_A: \Z^X\to \Z^A$ be the natural projection and let $G_A=\pi_A(G)$.
If $A$ and $B$ are finite subsets of $X$ such that $A\subseteq B$, then there exists a unique homomorphism $p^B_A: G_B\to G_A$ such that 
$\pi_A\restriction_G = p^B_A\circ \pi_A\restriction_G$.
Since $G$ is closed in $\Z^X$, it follows that $G$ is topologically isomorphic 
to the inverse limit of $G_A$'s with bounding maps $p^B_A$; see \cite[Corollary 2.5.7]{En}.
Clearly,
each $G_A$ is discrete as a subgroup of the discrete group $\Z^A$.
Since the latter group is torsion-free and finitely generated, so is $G_A$.
Since finitely generated torsion-free groups are free, each $G_A$ is free.
\end{proof}

\section{Proof of Theorem \ref{main:theorem}}

The proof of Theorem \ref{main:theorem} is split into a sequence of claims.

\begin{claim}
\label{continuity:of:evaluation}
For every $x\in \B$, the map $\varphi_x\in\Hom(G,\Z)$ defined by 
$\varphi_x(g)=g(x)$ for all $g\in G$, is continuous.
\end{claim}
\begin{proof}
This follows from the fact that $G$ is equipped with the topology of pointwise convergence on $X$. 
\end{proof}

\begin{claim}
\label{separate:continuity}
The evaluation map $e: \B\times G\to\Z$ defined by $e(x,g)=g(x)$ for $x\in \B$ and $g\in G$, is separately continuous.
\end{claim}
\begin{proof}
For a fixed $x\in \B$, the map $g\mapsto g(x)$ coincides with $\varphi_x$, so it is continuous by Claim \ref{continuity:of:evaluation}.

For a fixed $g\in G$, the map $x\mapsto g(x)$ coincides with $g$, so its continuity follows from Fact
\ref{aumatic:continuity}~(ii).
\end{proof}

\begin{claim}
\label{claim:a}
For every non-empty compact subset $K$ of $G$, there exists $m\in\N$ (depending on $K$) such that $g(O_m)\subseteq \{0\}$ for all $g\in K$, where
\begin{equation}
\label{eq:O_m}
O_m=\{x\in \B: x(i)=0 \mbox{ for all } i=1,2,\dots,m\}.
\end{equation}
\end{claim}
\begin{proof}
Let $K$ be a non-empty compact subset of $G$. Let $\varepsilon:\B\times K\to \Z$
be the restriction of the evaluation map $e$ to the subset $\B\times K$ of 
$\B\times G$. 
Then $\varepsilon$ is separately continuous by Claim
\ref{separate:continuity}.
Since $\B=\BS$ is a complete separable metric space and $K$ is compact, we can apply Namioka's theorem \cite{Namioka}
to find a dense $G_\delta$-subset $D$ of $\B$ such that
$\varepsilon$ is jointly continuous at every point of the set $D\times K$. 
Since $D\not=\emptyset$, we can choose $x_0\in D$.

For every $g\in K$, we can use the continuity of $\varepsilon$ at $(x_0,g)$ 
and the discreteness of $\Z$ to 
fix an open neighbourhood $U_g$ of $x_0$ in $\B$ and an open neighbourhood $V_g$ of $g$ in $K$ such that
$\varepsilon(U_g\times V_g)=\{\varepsilon(x_0,g)\}$.
Since $K$ is compact, there exist $g_1,g_2,\dots,g_k\in G$ such that
$K\subseteq \bigcup_{i=1}^k V_{g_i}$.
Now $U= \bigcap_{i=1}^k U_{g_i}$ is an open neighbourhood of $x_0$ in $\B$ such that
\begin{equation}
\label{eq:1}
g(x)=\varepsilon(x,g)=\varepsilon(x_0,g)=g(x_0)
\ 
\mbox{ for all }
\ 
x\in U
\ 
\mbox{ and each }
\ 
g\in K.
\end{equation}

Since $U$ is an open neighbourhood of $x_0$ in $\B=\BS$, there exists $m\in\N$
such that
\begin{equation}
\label{eq:2}
W=\{x\in \B: x(i)=x_0(i) \mbox{ for all } i=1,2,\dots,m\}\subseteq U.
\end{equation}

Let $x\in O_m$ and $g\in K$ be arbitrary. 
Define $y_0\in \B$ by 
$$
y_0(i)=\left\{\begin{array}{ll}
x_0(i) & \mbox{if $i=1,2,\dots,m$}\\
0 & \mbox{if $i>m$}
\end{array}
\right.
\ \ \
\mbox{for all}
\ \
i\in\N.
$$
Since $y_0\in W$ and $x+y_0\in W$, from
\eqref{eq:1} and \eqref{eq:2}
we obtain
$g(y_0)=g(x_0)$ and  
$g(x+y_0)=g(x_0)$.
Since $g$ is a homomorphism, $g(x+y_0)=g(x)+g(y_0)$. This gives $g(x)=0$.
We have proved that $g(O_m)\subseteq \{0\}$ for all $g\in K$.
\end{proof}

For each $n\in\N$, let $\delta_n\in X$ be the function defined 
by 
\begin{equation}
\label{eq:delta}
\delta_n(i)=\left\{\begin{array}{ll}
1 & \mbox{if $i=n$}\\
0 & \mbox{if $i\not=n$}
\end{array}
\right.
\ \ \
\mbox{for all}
\ \
i\in\N.
\end{equation}
Define 
\begin{equation}
\label{eq:psi:n}
\psi_n=\varphi_{\delta_n}\in \Hom(G,\Z).
\end{equation}
\begin{claim}
\label{claim:6}
For every non-empty compact subset $K$ of $G$, there exists $m\in\N$ (depending on $K$) such that
$\psi_n(K)= \{0\}$ for all $n\in\N$ with $n>m$.
\end{claim}
\begin{proof}
Let $m\in\N$ be as in the conclusion of Claim \ref{claim:a}.
If $n\in\N$ and $n>m$, then $\delta_n\in O_m$ by \eqref{eq:O_m} and \eqref{eq:delta}, so $\psi_n(g)=\varphi_{\delta_n}(g)=g(\delta_n)=0$ for all $g\in K$;
that is, $\psi_n(K)=\{0\}$.
\end{proof}

Let $\psi\in\Hom(G,\Z^\N)$ be the diagonal product of the family $\{\psi_n:n\in\N\}\subseteq \Hom(G,\Z)$ defined by 
$\psi(g)(n)=\psi_n(g)$ for all $n\in\N$ and $g\in G$; that is,
the $n$th coordinate of $\psi(g)\in\Z^\N$ is $\psi_n(g)$.

\begin{claim}
\label{psi:is:continuous}
$\psi$ is continuous.
\end{claim}
\begin{proof}
Each $\psi_n$ is continuous by \eqref{eq:psi:n} and Claim \ref{continuity:of:evaluation}.
It remains only to recall that the diagonal product of continuous maps is continuous.
\end{proof}

\begin{claim}
\label{psi:is:a:monomorphism}
$\psi$ is a monomorphism.
\end{claim}
\begin{proof}
Assume that $g\in G$ and $\psi(g)=0$. Then $g(\delta_n)=\psi_n(g)=\psi(g)(n)=0$
for all $n\in \N$. Since $g$ is a homomorphism from $\B$ to $\Z$, this implies 
$g(Y)=\{0\}$, where $Y$ is the subgroup of $\B$ generated by the set $\{\delta_n:n\in\N\}$. Since $Y$ is dense in $\B$ and $g$ is continuous by Fact
\ref{aumatic:continuity}~(ii), it follows that $g(X)=\{0\}$; that is, $g=0$.
We proved that $\psi$ has a trivial kernel, so it is a monomorphism.
\end{proof}

\begin{claim}
All compact subsets of $G$ are finite.
\end{claim}
\begin{proof}
Let $K$ be a non-empty compact subset of $G$. Choose $m$ as in the conclusion of Claim \ref{claim:6}.
Then 
$$
\psi(K)\subseteq \Z^m\times\{0\}\times\{0\}\times\dots\times\{0\}\times\dots,
$$
so $\psi(K)$ is discrete. On the other hand, $\psi$ is continuous by Claim \ref{psi:is:continuous}, and so $\psi(K)$ is compact. Therefore, $\psi(K)$ is finite.
Since $\psi$ is a one-to-one map by Claim \ref{psi:is:a:monomorphism}, it follows that $K$ is finite as well.
\end{proof}

\end{document}